\documentclass[11pt]{amsart}
\newcommand{\uml}{\mathcal{UML}}
\newcommand{\ml}{\mathcal{ML}}
\newcommand{\mf}{\mathcal{MF}}
\newcommand{\ah}{\mathrm{a.h.}}

\newcommand{\glt}{\mathcal{GL}_T}
\newcommand{\ie}{i.e.\ }
\newtheorem{lemma}{Lemma}
\newtheorem{theorem}{Theorem}
\newtheorem{corollary}[lemma]{Corollary}

\theoremstyle{definition}

\begin{document}

\title[Rigidity of unmeasured lamination spaces]{A note on the rigidity of unmeasured lamination spaces}
\author{Ken'ichi Ohshika}
\address{Department of Mathematics, Graduate School of Science, Osaka University, Toyonaka, Osaka 560-0043, Japan}
\email{ohshika@math.sci.osaka-u.ac.jp}
\date{}
\maketitle
\begin{abstract}
We show that every auto-homeomorphism of the unmeasured lamination space of an orientable surface of finite type is induced by a unique extended mapping class unless the surface is a sphere with at most four punctures or a torus with at most two punctures or a closed surface of genus $2$.
\end{abstract}
\section{Introduction}
The (extended) mapping class group of an orientable surface $S$ of finite type acts on spaces related to $S$ as automorphisms in various settings.
As typical examples, we should mention the Teichm\"{u}ller space of $S$, and the curve complex of $S$.
On the Teichm\"{u}ller space, the extended mapping class group acts by isometries.
Royden's theorem (\cite{Ro}) says that conversely every self-isometry of the Teichm\"{u}ller space is induced from a unique extended mapping class group unless $S$ is a sphere with at most four punctures or a torus with at most two punctures or a closed surface of genus $2$.
Similarly, Ivanov, Korkmaz and Luo (\cite{Iv, Ko, Lu}) showed that every simplicial automorphism of the curve complex of $S$ is induced from a unique extended mapping class with the same exceptions.

We can consider the same kind of rigidity for spaces of laminations  with various topologies.
The spaces of measured laminations and projective laminations on $S$ are topologically a Euclidean space and a sphere  respectively, and  their groups of homeomorphisms are much bigger than the extended mapping class group.
In contrast, the space of unmeasured laminations $\uml(S)$ with the quotient topology coming from the measured lamination space and the space of geodesic laminations with the Hausdorff topology or the Thurston topology have much fewer symmetries than the measured lamination space or the projective lamination space.
In fact, Papadopoulos proved in \cite{Pa} that there is a dense subset $\mathcal D$ of $\uml(S)$ such that for every auto-homeomorphism $f$ of $\uml(S)$ there exists a unique extended mapping class of $S$ which induces the same homeomorphism on $\mathcal D$ as $f$ with the same exceptions on $S$ as before; a sphere with less than five punctures, a torus with less than three punctures, and a closed surface of genus $2$.
(In \cite{Pa}, this result is stated for unmeasured foliations, but it is equivalent to the statement for unmeasured laminations as above.)
Charitos, Papadoperakis and Papadopoulos considered a similar problem in \cite{CPP}  for the geodesic lamination space with  with the Thurston topology $\glt(S)$.
They showed that every auto-homeomorphism on  $\glt(S)$ is induced from a unique extended mapping class of $S$ (with the same exceptions on $S$ as above).
Both of these results are based on the result of Ivanov,  Korkmaz  and Luo which we mentioned above.

In this note, we shall refine Papadopoulos's theorem cited above to obtain the following result, in which we do not restrict homeomorphisms to a dense subset any more:
Let $S$ be an orientable surface of finite type which is not one of the exceptions listed above, and $f: \uml(S) \rightarrow \uml(S)$  a homeomorphism.
Then, we shall show that there exists a unique extended mapping class $h$ which induces $f$ on $\uml(S)$.
(Even when $S$ is a closed surface of genus $2$, the existence of $h$ is still valid.)

The author would like to express his gratitude to Athanase Papadopoulos, whose work motivated the present result, and who kindly read through a draft of this note and gave valuable comments.

\section{Preliminaries}
Throughout this note, we assume that $S$ is an orientable surface of finite type which is neither  a sphere with at most four punctures, nor a torus with at most two punctures.

Fix some complete hyperbolic metric on $S$.
A geodesic lamination is a closed subset of $S$ consisting of disjoint simple geodesics.
A transverse measure of a geodesic lamination is a measure defined on arcs transverse to leaves of the lamination, which is invariant under homotopies along leaves.
A geodesic lamination equipped with a transverse measure is called a measured lamination.
We always assume that the support of transverse measure is equal to the entire lamination.
We  endow a weak topology on the set of measured laminations on $S$ and denote it by $\ml(S)$.
It is known that for two hyperbolic metrics $m_1, m_2$ on $S$, there is a one-to-one correspondence between the geodesic laminations on $(S,m_1)$ and those on $(S,m_2)$, and $\ml(S)$ does not depend on the hyperbolic metric which we fixed at the beginning.
Also as was observed by Thurston and Levitt \cite{Le}, there is a one-to-one correspondence between the measured laminations and the measured foliations and $\ml(S)$ is homeomorphic to the measured foliation space $\mf(S)$ defined by Thurston (see \cite{FLP, Th}).
Therefore, all the results in this note also hold even if we change measured laminations to measured foliations.

The quotient space of $\ml(S)$ obtained by forgetting transverse measures is called the unmeasured lamination space, and is denoted by $\uml(S)$.
This space is all the more important as its subset consisting of maximal and minimal laminations constitutes the Gromov boundary of the curve complex of $S$, (see Klarreich \cite{Kl}).

For a homeomorphism $f$ on $S$ and a measured lamination $\lambda$, by defining $f(\lambda)$ to be the geodesic lamination homotopic to $f(\lambda)$ with the transverse measure obtained by pushing forward that of $\lambda$, we can regard $f$ as acting on $\ml(S)$.
Furthermore, it is easy to see that this action is homeomorphic, and that the action depends only on the homotopy class of $f$.
Therefore, the extended mapping class group acts on $\ml(S)$ as a group of homeomorphisms.
By taking quotient, we can also regard the extended mapping class group as acting on $\uml(S)$ by homeomorphisms.

\section{Results}
The unmeasured lamination space $\uml(S)$ is not $T_1$, that is, there may be a point which is contained in every neighbourhood of another point.
In the following lemma, we shall see for which pairs of points  this phenomenon of non-$T_1$ occurs.

We say that $F \in \uml(S)$ is unilaterally adherent to $G \in \uml(S)$ when every neighbourhood of $G$ contains $F$ whereas $G$ and $F$ are distinct.

\begin{lemma}
\label{unilaterally adherent}
The following two conditions are equivalent for two distinct unmeasured laminations $F, G \in \uml(S)$.
\begin{enumerate}
\item $G$ is contained in $F$, \ie $G$ is a sublamination of $F$.
\item $F$ is unilaterally adherent to $G$.
\end{enumerate}
\end{lemma}
\begin{proof}
First, we shall show the first condition implies the second.
This is the same argument as in the proof of Lemma 3.1 in Papadopoulos \cite{Pa}.
Suppose that $G$ is contained in $F$.
Let $\mathcal G$ be a measured lamination representing $G$.
We can take a measured lamination $\mathcal F$ containing $\mathcal G$ as a representative of $F$.
Then, we can consider a family of measured laminations $\mathcal G \sqcup t(\mathcal F \setminus \mathcal G)$ with $t >0$, which is a representative of $F$  and converges to $\mathcal G$ in $\ml(S)$ as $t \rightarrow 0$.
Therefore, every neighbourhood of $G$ contains $F$.

Next we shall show that the second condition implies the first.
Suppose that the first condition does not hold.
Then  there is a component $G_0$ of $G$ which either intersects a component $F_0$ of $F$ transversely or is disjoint from $F$.
Suppose first that $G_0$ intersects $F_0$ transversely.
Fix some transverse measure on $F_0$ and let $\mathcal F_0$ the measured lamination which is $F_0$ endowed with the transverse measure.
Now consider an open set $U_{F_0}$ in $\ml(S)$ which is defined by $U_{F_0}=\{\mathcal H\in \ml(S) \mid i(\mathcal H,\mathcal F_0) >0\}$.
Then $\mathcal G \in U_{F_0}$, whereas $\mathcal F \not\in U_{F_0}$.
Since $\pi^{-1}\pi (U_{F_0})=U_{F_0}$, by the definition of the topology on $\uml(S)$, we see that $\pi(U_{F_0})$ is an open set containing $G$ but not $F$.
This shows that $F$ is not unilaterally adherent to $G$.

Suppose next that $G_0$ is disjoint from $F$.
Then, there is a simple closed curve $\gamma$ intersecting $G_0$ essentially which is disjoint from $F$.
We define an open set $U_\gamma$ to be $\{\mathcal H \in \ml(S) \mid i(\mathcal H, \gamma)>0\}$.
Then by the same argument as above $\pi(U_\gamma)$ is an open set containing $G$ but not $F$.
Therefore $F$ is not unilaterally adherent to $G$ also in this case.
\end{proof}

\begin{lemma}
\label{good sequence}
Let $F$ be an unmeasured lamination.
Let  $\{K_i\}$ be a sequence of multi-curves which converges to $F$ both in $\uml(S)$ and in the Hausdorff topology.
Let  $G$ be an unmeasured lamination distinct from $F$,  and suppose that $F$ is not unilaterally adherent to $G$.
Then  $\{K_i\}$ does not converge to $G$ in $\uml(S)$.
\end{lemma}
\begin{proof}
Consider laminations $F$ and $G$ in $\uml(S)$ as in the statement.
Then by Lemma \ref{unilaterally adherent}, $G$ is not contained in $F$.
Let $\mathcal F$ and $\mathcal G$ be measured laminations representing $F$ and $G$ respectively.
Let $\{K_i\}$ be a sequence of multi-curves as in the statement, and $\{\mathcal K_i\}$ be a sequence of weighted multi-curves representing $K_i$ which converges to $\mathcal F$.

Suppose first that $G$ has a component  which intersects $F$ transversely.
Then we have $i(\mathcal F, \mathcal G) >0$.
Lemma 3.1 in Papadopoulos \cite{Pa} shows that there are neighbourhoods $U_F$ of $F$ and $U_G$ of $G$ in $\uml(S)$, which are disjoint.
This implies that a sequence converging to $F$ cannot converge to $G$.

Next suppose that $G$ has a component $G_0$ which is disjoint from $F$.
Then as in the proof of Lemma \ref{unilaterally adherent}, there is a simple closed curve $\gamma$ which intersects $G_0$ essentially but is disjoint from $F$.
Let $U_\gamma$ be the open set in $\ml(S)$ containing $\mathcal G$, which was defined there as $U_\gamma=\{\mathcal H \in \ml(S) \mid i(\mathcal H, \gamma)>0\}$.
Then, since $\{K_i\}$ converges to $F$ in the Hausdorff topology, $K_i$ is disjoint from $\gamma$ for sufficiently large $i$.
Therefore $\mathcal K_i$ is not contained in $U_\gamma$ for sufficiently large $i$.
This implies that $\{K_i\}$ cannot converge to $G$.
\end{proof}

For an unmeasured lamination $F$, we define, as in \cite{Oh}, its adherence height to be the maximal length $m$ of sequences $(F=F_0, \dots, F_m)$ in $\uml(S)$ such that $F_j$ is  unilaterally adherent to $F_{j+1}$, that is, $F_{j+1}$ is a proper sublamination of $F_j$ by Lemma \ref{unilaterally adherent}.
We call such a sequence an adherence tower.
We denote the adherence height of $F$ by $\ah(F)$.
We note that there is an upper bound for adherence heights depending only on $S$.
It is obvious that any auto-homeomorphism of $\uml(S)$ preserves adherence heights.

\begin{lemma}
\label{strictly less}
Suppose that $F \in \uml(S)$ is unilaterally adherent to $G$.
Then we have $\ah(G) < \ah(F)$.
\end{lemma}
\begin{proof}
Consider an adherence tower $(G_0=G, \dots, G_m)$ realising the adherence height of $G$.
Since $F$ is unilaterally adherent to $G$, the sequence $(F, G_0, \dots , G_m)$ is also an adherence tower, which shows that $\ah(F) \geq \ah(G)+1$.
\end{proof}

\begin{theorem}
\label{main}
Let $h: \uml(S) \rightarrow \uml(S)$ be a homeomorphism.
Then there exists an extended mapping class $f$ such that $f=h$ as actions on $\uml(S)$.
If $S$ is not a closed surface of genus $2$, then such $f$ is unique.
\end{theorem}
\begin{proof}
By the main result of Papadopoulos \cite{Pa}, for any given auto-ho\-meo\-mor\-phism $h$, there is an extended mapping class $f$ such that $f$ and $h$ coincide in the subset of $\uml(S)$ consisting of multi-curves.
What we have to show is that for this $f$, we have $f=h$ on the entire $\uml(S)$.

Let $F$ be an unmeasured lamination which is not a multi-curve.
Take a sequence of multi-curves $\{K_i\}$ as in Lemma \ref{good sequence}.
Since $f$ and $h$ coincide on the set of multi-curves, we see that $\{f(K_i)=h(K_i)\}$ converges to both $f(F)$ and $h(F)$.
Since $\{f(K_i)\}$ converges to $f(F)$ in the Hausdorff topology, by Lemma \ref{good sequence}, either $h(F)$ coincides with $f(F)$ or $f(F)$ is unilaterally adherent to $h(F)$.
Suppose that the latter is the case.
Then by Lemma \ref{strictly less} we have $\ah (h(F)) < \ah (f(F))$.
Since auto-homeomorphisms of $\uml(S)$ preserves  adherence heights, we see that $\ah(h(F))=\ah(F)=\ah(f(F))$.
This is a contradiction.
Thus we have shown that $f(F)=h(F)$ for any $F$ in $\uml(S)$.

The uniqueness follows from Theorem 1.1-(2) in Papadopoulos \cite{Pa}.
\end{proof}

We get the following corollary from this theorem.

\begin{corollary}
\label{auto group}
Suppose that $S$ is not a closed surface of genus $2$.
Then the group of auto-homeomorphisms of $\uml(S)$ coincides with the extended mapping class group of $S$.
\end{corollary}

\end{document}